\documentclass[11pt]{amsart}
\usepackage{amsmath, amssymb, amsbsy, amsfonts, amsthm, latexsym, amsopn, amstext, amsxtra, euscript, amscd, color, mathrsfs}
\usepackage{amsmath,amscd}
\usepackage[normalem]{ulem}
\usepackage{soul}

\makeatletter
\@namedef{subjclassname@2020}{%
  \textup{2020} Mathematics Subject Classification}
\makeatother
\PassOptionsToPackage{hyphens}{url}
\usepackage[hidelinks]{hyperref}

\allowdisplaybreaks

\usepackage{tabularx}
\usepackage[colorinlistoftodos,prependcaption,textsize=tiny]{todonotes}

\usepackage{amscd}
\usepackage{color,enumerate}
\usepackage{array, multirow, multicol}
\usepackage{longtable}
\usepackage{bm}  
\newcommand{\RNum}[1]{\lowercase\expandafter{\romannumeral #1\relax}}

\newtheorem{thm}{Theorem}[section]

\newtheorem{cor}[thm]{Corollary}
\newtheorem{prop}[thm]{Proposition}

\newtheorem{rmk}[thm]{Remark}

\newtheorem{thm-con}[thm]{Theorem-Conjecture}
\numberwithin{equation}{section}

\theoremstyle{definition}
\newtheorem{defn}[thm]{Definition}

\newcommand{\f}{\Bbb F}

\newcommand{\cA}{\mathcal A}

\newcommand{\cV}{\mathcal V}

\newcommand{\cS}{\mathcal S}
\newcommand{\cU}{\mathcal U}
\newcommand{\cW}{\mathcal W}

\newcommand{\cPS}{\mathcal{PS}}

\newcommand{\F}{{\mathbb F}}
\newcommand{\Tr}{{\operatorname{Tr}}}

\begin{document}
\title[bent partitions]{General constructions of normal bent partitions related to vectorial dual-bent functions}

\author[A. Kholosha]{Alexander Kholosha}
\address{The Selmer Center, Department of Informatics, University of Bergen, PB 7803, N-5020, Bergen, Norway}
\email{oleksandr.kholosha@uib.no}

\author[M. Pal]{Mohit Pal}
\address{The Selmer Center, Department of Informatics, University of Bergen, PB 7803, N-5020, Bergen, Norway}
\email{mathmohit@outlook.com}

\maketitle 

\begin{abstract}
    Bent partitions were introduced as a generalization of partial spread construction of bent functions and they became a hot research 
    topic recently. In this paper, we prove two general constructions of normal bent partitions that are related to vectorial dual-bent 
    functions. They cover many of the currently known bent partitions of this type as partial cases. These constructions also provide a 
    large number of new bent partitions. Relevant properties of bent functions obtained from such partitions are proven. 
\end{abstract}

\section{Introduction}

Boolean bent functions were first introduced by Rothaus~\cite{R76} in 1976 as an interesting combinatorial object. They are functions from $\F_2^n$ to $\F_2$ with $n$ even, whose Hamming distance to the set of all affine functions is maximum. Equivalently, the absolute square of their Walsh transform coefficients is constant. Later, the research in this area was stimulated by the significant relation to the following topics in mathematics and computer science: algebraic combinatorics, coding theory, sequences and cryptography. Kumar, Scholtz and Welch in~\cite{KSW85} generalized the notion of Boolean bent functions to the case of functions from $(\mathbb Z/q\mathbb Z)^n$ to $\mathbb Z/q\mathbb Z$, where $q$ is the power of a prime $p$. Complete classification of bent functions looks hopeless even in the binary case and the case of $p$-ary bent functions is still more complicated. However, many explicit methods are known for constructing bent functions, either from scratch (with so-called primary constructions) or from other bent functions (in so-called secondary constructions). 

Two major primary constructions of bent functions are found in the original Maiorana-McFarland and the partial spread classes. The original Maiorana-McFarland bent functions from $\F_q \times \F_q$ to $\F_p$ (where $q=p^m$) are functions of the form 
\begin{equation}\label{MMFform}
f(x,y)= \Tr_m(y \pi(x)) + h(x)
\end{equation}
for a permutation $\pi$ of $\F_q$ and a function $h$ from $\F_q$ to $\F_p$. It is well known that $\pi$ being a permutation is necessary 
and sufficient for $f$ to be bent. Any function $f:\F_q \times \F_q\mapsto\F_p$ that is affine on the cosets of a $m$-dimensional subspace 
$U$ of $\F_{q^2}$ is affine equivalent to a function of the form (\ref{MMFform}). We say then that $f$ belongs to the completed original 
class of Maiorana-McFarland bent functions {\em with associated subspace $U$}. Partial spread bent functions are bent functions whose 
preimage sets are formed by unions of $m$-dimensional subspaces of $\F_{q^2}$. Both of these classes contain only regular bent functions. 

Let $\cS_{s,k}= \{U_0, U_1, \ldots,U_s\}$, where $U_i\subset \F_{q^2}$, $\lvert U_i\rvert= p^k$ for $0 \leq i \leq s$ and $U_i \cap U_j = \{0\}$ for all  $0 \leq i<j \leq s$. We call $\cS_{s,k}$ to be a {\bf partial $(s,k)$-partition} of $\F_{q^2}$. When $k$ divides $2m$ and $s=\frac{q^2-1}{p^k-1}-1$, we call  $\cS_{s,k}$ a {\bf $k$-partition} of $\F_{q^2}$ and denote it by $\cS_{k}$. In addition, if $U_i$'s are $k$-dimensional subspaces of $\F_p^{2m}$ for all $0 \leq i \leq s$, we call $\cS_{s,k}$ a {\bf partial $k$-spread} of $\F_p^{2m}$ and $\cS_{k}$ a {\bf $k$-spread} of $\F_p^{2m}$. In particular, if $k=m$, then these are called {\bf partial spread} and {\bf spread} of $\F_p^{2m}$. 

For any function $f:\F_{p^n}\mapsto \F_p$ and $j \in \F_p$, denote 
\[D_j := \{x \in \F_{p^n} \mid f(x)=j\}.\] 
Assume $\cS_{s,m}$ is a partial spread of $\F_{q^2}$ and $f$ is such that, $D_j \cup \{0\}$ is a union of exactly $p^{m-1}$ subspaces $U_i$ 
for every $j \in \F^*_p$ and $D_0=\F_p^{2m}\setminus\bigcup^s_{i=0}U_i^*$. Then $f$ is bent and every bent function which is constant on 
every $U_i^*$ must be of this form. This class of bent functions is called the {\bf partial spread class}. 

Motivated by the fact that the bentness is a property of the partition, Anbar and Meidl~\cite{AM22}, introduced the notion of bent partitions as follows. 

\begin{defn}
A partition $\Omega  = \{A_0,A_1,\ldots,A_K \}$ of $\F_{p^n}$ is called a bent partition of $\F_{p^n}$ of depth $K$ if every function $f:\F_{p^n}\mapsto \F_p$ such that $f(A_0)\equiv c\in \F_p$ and for every $j\in\F_p$, $D_j$ contains exactly $K/p$ elements of $\Omega \backslash \{A_0\}$, is bent. 
\end{defn}
A trivial example of the bent partition is a partition of $\F_{3^n}$ obtained from a bent function $f:\F_{3^n}\mapsto \F_3$ as $\Omega =\{D_0, D_1, D_2\}$. In~\cite[Remark 6]{AAKM24}, the authors showed that except for this example, bent partitions can exist only for $\F_{p^n}$ with $n=2m$. In~\cite[Theorem 3]{WWF26}, the authors showed that the depth $K$ of a bent partition $\Omega$ of $\F_{q^2}$, for which all bent functions generated by $\Omega$ are regular or all are weakly regular but not regular, is always a power of $p$. Based on the cardinalities of the sets $A_i$'s the authors divided bent partitions into the following two types: 

\begin{itemize}
    \item Type I. $\lvert A_0 \rvert=q$ and $\lvert A_1 \rvert = \lvert A_2 \rvert = \cdots = \lvert A_K \rvert$ 
    \item Type II. $\lvert A_0 \rvert=\lvert A_1 \rvert-q$  and $\lvert A_1 \rvert = \lvert A_2 \rvert = \cdots = \lvert A_K \rvert$  
\end{itemize}
All bent partitions with $K>4$, known so far, are of Type I with $A_0$ being an $m$-dimensional subspace. These are known as normal bent partitions, that are defined as follows: 

\begin{defn}\label{nbp}
Normal partition $\Omega  = \{U, A_1,\ldots,A_K \}$ of $\F_{q^2}$, where $U$ is an $m$-dimensional subspace of $\F_{q^2}$, is called a normal bent partition of $\F_{q^2}$ of depth $K$ if every function $f:\F_{q^2}\mapsto\F_p$ such that $f(U)\equiv c\in \F_p$ and for every $j\in\F_p$, $D_j$ contains exactly $K/p$ elements of $\Omega \backslash \{U\}$, is bent. 
\end{defn}
In this paper, we start with $m$-partitions of $\F_{q^2}$. If exactly one of the elements of an $m$-partition is a subspace of $\F_{q^2}$, then we call it a {\bf normal $m$-partition}. Then we combine the elements of $m$-partitions to form {\bf partitions} of $\F_{q^2}$. In the definition of a parition, there is no restriction on the size of the sets, the only property required is that the sets are disjoint and their union is equal to $\F_{q^2}$. Similarly, if exactly one of the elements of a partition is a subspace of $\F_{q^2}$, then we call it a {\bf normal partition}. We use permutation polynomials over $\F_q$ and normal $m$-partitions of $\F_{q^2}$ to construct normal bent partitions of $\F_{q^2}$. Choosing particular permutation polynomials allows our normal bent partition to cover all the normal bent partitions known so far.

\section{Normal Bent Partitions of Type I (Construction 1)}
Let $p$ be an odd prime and $m,k$ are positive integers such that $q=p^m$ and $k$ divides $m$. Let $g_1$ be a function from $\F_{q^2}$ to 
$\F_q$ given by 
\[
g_1(z)= \frac{z-z^q}{\beta} (\pi_1(z+z^q))^{q-2},
\]
where $\pi_1$ is a permutation over $\F_q$ such that $\pi_1(0)=0$ and $\beta \in \F_{q^2}$ such that $\beta^q=-\beta$. Let
\[
U:= \{ z \in \F_{q^2} \mid z+z^q=0\} = \beta \F_q.
\]
Then it is easy to observe that $\lvert U \rvert =q$ and $g_1(U)\equiv0$. For any element $j \in \F_q$, define the set 
\[
U_j := \left \{z \in \F_{q^2} \mid z-z^q=  j \beta \pi_1(z+z^q)  \right \},
\]
and $U_j^*=  U_j \backslash \{0\}$. It is easy to observe that $U \cap U_j = \{0\}$ for all $j \in \F_q$ and $g_1(U_j^*)\equiv j$. Also, notice that for any $j \in \F_q$, the map 
\begin{equation}\label{u1}
u \mapsto \frac{1}{2}(u+j\beta \pi_1(u))
\end{equation}
is a bijection from $\F_q$ to $U_j$. Thus, $\lvert U_j \rvert = q$ for all $j \in \F_q$ and $U_{j_1} \cap U_{j_2}= \{0\}$ for all $j_1, j_2 
\in \F_q$, $j_1 \neq j_2$. Thus, the set 
\[
\cU_m= \{U, U_j \mid j \in \F_q\}
\]
is a normal $m$-partition of $\F_{q^2}$. Now we combine the elements of $\cU_m$ to form a normal partition of $\F_{q^2}$. For each $\gamma 
\in \F_{p^k}$, define the sets 
\[
\cA_1(\gamma):= \bigcup_{\substack{j \in \F_q \\ \Tr_k^m(j)= \gamma}} U_j^*.  
\]
Thus, the set
\begin{equation}\label{part1}
    \cS_1=\{U, \cA_1(\gamma) \mid \gamma \in \F_{p^k}\}
\end{equation}
is a normal partition of $\F_{q^2}$. One may note that $\lvert \cA_1(\gamma) \rvert= p^{m-k}(q-1)$ for all $\gamma \in \F_{p^k}$. 

Define function $f_1:\F_{q^2}\mapsto\F_p$ as $f_1(z) = B(\Tr_k^m(g_1(z)))+h_1(z+z^q)$, where $B:\F_{p^k}\mapsto\F_p$ is any balanced 
function with $B(0)=0$ and $h_1:\F_q\mapsto \F_p$. Take any $z\in U_j^*$ with $j\in\F_q$, then $f_1(z) = B(\Tr_k^m(j))+h_1(u)$ for some 
unique $u\in\F^*_q$ (because of mapping in (\ref{u1})). Thus, $f_1$ is constant on $U_j^*$ if and only if $h_1$ is constant on $\F^*_q$. We 
need $f_1$ to be constant on the elements of $\cS_1$ and wlog can assume 
\begin{equation}\label{f1}
\begin{split} 
    f_1(z) &= B(\Tr_k^m(g_1(z)))+ c(1-(z+z^q)^{q-1})= \begin{cases}
         c~&~\mbox{if}~z\in U\\
         B(\Tr_k^m(g_1(z)))~~&~\mbox{if}~z\in \F_{q^2} \backslash U,
     \end{cases}
\end{split}
\end{equation}
where $c \in \F_p$. 

\begin{thm}\label{th:f1}
For any function $f_1$ given in~\eqref{f1}, and any $a \in \F_q$, the corresponding Walsh transform coefficient $\cW_{f_1}(a)=q 
\zeta_p^{c}$ and, thus, $f^*_1(a)=c$. Further, for any $a \in \F_{q^2} \backslash \F_q$, consider two cases. 
\renewcommand{\theenumi}{\roman{enumi}}
\renewcommand{\labelenumi}{(\theenumi)}
\begin{enumerate}
\item\label{it:1} If $B()$ is linear on $\F_{p^k}$ then $f_1$ is bent if and only if $B(\gamma)=\Tr_k(t'\gamma)$ for some 
    $t'\in\F^*_{p^k}$. In this case, $f_1$ is a regular bent function and $f^*_1(-a)=\Tr_m \left(a' \pi_1^{-1}\left(\frac{-t'}{b'}\right) 
    \right)$. Moreover, vectorial function $F_1(z)=\Tr_k^m(g_1(z))+ c(1-(z+z^q)^{q-1})$ is bent and all its components belong to the 
    completed original class of Maiorana-McFarland bent functions with associated subspace $U$. 
\item\label{it:2} If $\pi_1^{-1}(x)$, considered as a polynomial in $\F_q[x]$, is such that $\pi_1^{-1}(x) \pmod{x^{p^k}-x}$ is a 
    linearized polynomial over $\F_q[x]$ (say, $\pi_1^{-1}(x) = \sum_{i} \alpha_i x^{s_i(p^k-1)+p^{r_i}}$) then $f_1$ is a regular bent 
    function and $f^*_1(-a)=B(-\delta_1)$, where $\delta_1= \Tr_k^m \left(\sum_{i} \left(\frac{a'\alpha_i}{ 
    (b')^{s_i(p^k-1)+p^{r_i}}}\right)^{p^{k-r_i}}\right)$. 
\end{enumerate}
In both cases, $a' =\frac{a+a^q}{2} \in \F_q $ and $b' = \frac{\beta(a-a^q)}{2} \in \F_q$.
\end{thm}

\begin{proof}
Start with calculating the Walsh transform at $a=0$. 
\begin{align*}
        \cW_{f_1}(0) &= \sum_{z \in \F_{q^2}} \zeta_p^{B(\Tr_k^m(g_1(z)))+ c(1-(z+z^q)^{q-1})}\\
        &=  \sum_{z \in U} \zeta_p^c + \sum_{\gamma \in \F_{p^k}} \sum_{z \in \cA_1(\gamma)} \zeta_p^{B(\Tr_k^m(g_1(z)))}\\
        &= \lvert U \rvert \zeta_p^{c} + \sum_{\gamma \in \F_{p^k}}  \lvert \cA_1(\gamma) \rvert \zeta_p^{B (\gamma)} \\
        &=q \zeta_p^{c} + p^{m-k}(q-1) \sum_{\gamma \in \F_{p^k}} \zeta_p^{B (\gamma)}\\
        &=q \zeta_p^{c}.
    \end{align*} 
Now, for any nonzero $a \in \F_{q^2}$, we have
\begin{align*}
     \cW_{f_1}(-a) &= \sum_{z \in \F_{q^2}} \zeta_p^{B(\Tr_k^m(g_1(z)))+ c(1-(z+z^q)^{q-1}) +\Tr_{2m}(az)}\\
        &=  \sum_{z \in U} \zeta_p^{c+\Tr_{2m}(az)} + \sum_{\gamma \in \F_{p^k}} \sum_{z \in \cA_1(\gamma)} \zeta_p^{B(\Tr_k^m  (g_1(z)))+\Tr_{2m}(az)}\\
        &= \zeta_p^c \sum_{z \in U} \zeta_p^{\Tr_{2m}(az)} + \sum_{\gamma \in \F_{p^k}} \zeta_p^{B(\gamma)} \sum_{z \in \cA_1(\gamma )} \zeta_p^{\Tr_{2m}(az)} \\
     &= \zeta_p^c \sum_{z \in U} \zeta_p^{\Tr_{2m}(az)} + \sum_{\gamma \in \F_{p^k}} \zeta_p^{B(\gamma)} \sum_{\substack{j \in \F_q \\ \Tr_k^m(j)=\gamma }}  \sum_{z \in U_j^*} \zeta_p^{\Tr_{2m}(az)} \\
     &\stackrel{\eqref{u1}}{=} \zeta_p^c \sum_{z \in U} \zeta_p^{\Tr_{2m}(az)} + \sum_{\gamma \in \F_{p^k}} \zeta_p^{B(\gamma)}  \sum_{\substack{j \in \F_q \\ \Tr_k^m(j)=\gamma }}  \sum_{u \in \F_q^*} \zeta_p^{\Tr_{2m}\left(\frac{a}{2}(u+j\beta \pi_1(u))\right)} \\
     &= \zeta_p^c \sum_{z \in U} \zeta_p^{\Tr_{2m}(az)} + \sum_{\gamma \in \F_{p^k}} \zeta_p^{B(\gamma)}  \sum_{\substack{j \in \F_q \\ \Tr_k^m(j)=\gamma }}  \sum_{u \in \F_q^*} \zeta_p^{\Tr_m \left(u \left(\frac{a+a^q}{2} \right)+j \beta \pi_1(u) \left(\frac{a-a^q}{2} \right) \right)}\\
    &= \zeta_p^c \sum_{z \in U} \zeta_p^{\Tr_{2m}(az)} + \sum_{\gamma \in \F_{p^k}} \zeta_p^{B(\gamma)}  \sum_{\substack{j \in \F_q \\ \Tr_k^m(j)=\gamma }}  \sum_{u \in \F_q^*} \zeta_p^{\Tr_m \left(a' u+jb'\pi_1(u) \right)},
\end{align*}
where $\frac{a+a^q}{2}=a' \in \F_q $ and $\frac{\beta(a-a^q)}{2} =b' \in \F_q$. Now, we consider two cases, namely, $a \in \F_q^*$ and $a \in \F_{q^2} \backslash \f_q$.\\

\textbf{Case 1.} Let $a \in \F_q^*$. In this case, $a'=a$, $b'=0$ and hence we have
\begin{align*}
    \cW_{f_1}(-a)  &= \zeta_p^c \sum_{z \in U} \zeta_p^{\Tr_{2m}(az)} + \sum_{\gamma \in \F_{p^k}} \zeta_p^{B(\gamma)}  \sum_{\substack{j \in \F_q \\ \Tr_k^m(j)=\gamma }}  \sum_{u \in \F_q^*} \zeta_p^{\Tr_m (au )}\\
    &= \zeta_p^c \sum_{z \in U} \zeta_p^{\Tr_m(a(z+z^q))} - p^{m-k} \sum_{\gamma \in \F_{p^k}} \zeta_p^{B(\gamma)}\\
    &=q \zeta_p^c.
\end{align*}

\textbf{Case 2.} Let $a \in \F_{q^2} \backslash \f_q$. In this case, we have $b' \neq 0$. For each $\gamma \in \F_{p^k}$, fix $j_\gamma \in 
\F_q$ such that $\Tr_k^m(j_\gamma)= \gamma $. Let $j= j_\gamma + v\in\F_q$ for some $v \in \ker(\Tr_k^m)$. Then, we have 
\begin{align*}
    \allowdisplaybreaks 
    \cW_{f_1}(-a) &=   \zeta_p^c \sum_{z \in U} \zeta_p^{\Tr_{2m}(az)} + \sum_{\gamma \in \F_{p^k}} \zeta_p^{B(\gamma)} \sum_{\substack{j \in \F_q \\ \Tr_k^m(j)=\gamma }}  \sum_{u \in \F_q^*} \zeta_p^{\Tr_m \left(a' u+jb' \pi_1(u) \right)}\\
      &= \zeta_p^c \sum_{z \in \F_q} \zeta_p^{\Tr_{2m}(a \beta z)} + \sum_{\gamma \in \F_{p^k}} \zeta_p^{B(\gamma)}  \sum_{v \in \ker(\Tr_k^m)}  \sum_{u \in \F_q^*} \zeta_p^{\Tr_m \left(a' u+(j_{\gamma}+v)b' \pi_1(u)  \right)} \\
      &= \zeta_p^c \sum_{z \in \F_q} \zeta_p^{\Tr_m((a-a^q)\beta z)} + \sum_{\gamma \in \F_{p^k}} \zeta_p^{B(\gamma)}  \sum_{v \in \ker(\Tr_k^m)}  \sum_{u \in \F_q^*} \zeta_p^{\Tr_m \left(a' u+j_{\gamma}b' \pi_1(u) \right)} \zeta_p^{\Tr_m \left(b' v \pi_1(u) \right)} \\
        &=  \sum_{\gamma \in \F_{p^k}} \zeta_p^{B(\gamma)} \sum_{u \in \F_q^*} \zeta_p^{\Tr_m \left(a' u+j_{\gamma} b' \pi_1(u) \right)}  \sum_{v \in \ker(\Tr_k^m)}  \zeta_p^{\Tr_m \left(b' v \pi_1(u) \right)} \\
        &=  \sum_{\gamma \in \F_{p^k}} \zeta_p^{B(\gamma)} \sum_{u \in \F_q^*} \zeta_p^{\Tr_m \left(a' u+j_{\gamma} b' \pi_1(u) \right)}  \sum_{v \in \ker(\Tr_k^m)}  \zeta_p^{\Tr_k \left( \Tr_k^{m} \left(b' v \pi_1(u) \right) \right)}\\
        &=p^{m-k} \sum_{\gamma \in \F_{p^k}} \zeta_p^{B(\gamma)}  \sum_{\substack{u \in \F_q^* \\ b' \pi_1(u) \in \F_{p^k}^*}} \zeta_p^{\Tr_m \left(a' u+j_{\gamma} b' \pi_1(u) \right)} \\
         &= p^{m-k} \sum_{\gamma \in \F_{p^k}} \zeta_p^{B(\gamma)} \sum_{\substack{u \in \F_q^* \\ b' \pi_1(u)  \in \F_{p^k}^*}} \zeta_p^{\Tr_m (a' u)+\Tr_k (\gamma b' \pi_1(u) )}.
    \end{align*}
Since $\pi_1$ is bijective on $\F_q$, the map $u \mapsto b'\pi_1(u)$ is bijective on $\F_q$. Thus, there exist precisely $p^k-1$ element $u 
\in \F_q^*$ such that $b'\pi_1(u) \in \F_{p^k}^*$. Let $b' \pi_1(u)= t$ then $u=\pi_1^{-1}\left(\frac{t}{b'}\right)$, where $\pi_1^{-1}$ is 
the compositional inverse of $\pi_1$. Thus, we have 
\begin{equation}\label{MainEqT1}
    \cW_{f_1}(-a) = p^{m-k} \sum_{\gamma \in \F_{p^k}} \zeta_p^{B(\gamma)} \sum_{t \in \F_{p^k}^*} \zeta_p^{\Tr_m \left(a' \pi_1^{-1}\left(\frac{t}{b'}\right) \right)+\Tr_k (\gamma t)}.
    \end{equation}

In particular, if $B$ is linear on $\F_{p^k}$ then
\begin{align*}
\cW_{f_1}(-a)&=p^{m-k} \sum_{t \in \F_{p^k}^*} \zeta_p^{\Tr_m \left(a' \pi_1^{-1}\left(\frac{t}{b'}\right) \right)}\sum_{\gamma \in \F_{p^k}} \zeta_p^{B(\gamma)+\Tr_k (\gamma t)}\\
&=q \sum_{\substack{t \in \F_{p^k}^*\\B(\gamma)\equiv -\Tr_k (\gamma t)}} \zeta_p^{\Tr_m \left(a' \pi_1^{-1}\left(\frac{t}{b'}\right) \right)}\\
&=q \zeta_p^{\Tr_m \left(a' \pi_1^{-1}\left(\frac{-t'}{b'}\right) \right)},
\end{align*}
if $B(\gamma)=\Tr_k(t'\gamma)$ for some $t'\in\F^*_{p^k}$ and $\gamma\in\F_{p^k}$. Obviously, $|\cW_{f_1}(-a)|\neq q$ if $B$ is not of the 
trace form. From \eqref{MainEqT1} it also follows that taking $B(\gamma)=\Tr_k(t'\gamma)$ for any $t'\in\F^*_{p^k}$ we have 
$\cW_{f_1}(-a)=q \zeta_p^{\Tr_m \left(a' \pi_1^{-1}\left(\frac{-t'}{b'}\right) \right)}$. Therefore, vectorial function $F_1(z)= 
\Tr_k^m(g_1(z))+ c(1-(z+z^q)^{q-1})$ is bent and all its components are bent (each component function $\Tr_k(t'F_1(z))$ has the form of 
\eqref{f1} with $c\Tr_k(t')$ in place of $c$). 

In the general case with balanced $B$, consider $\pi_1^{-1}(x)$ as a polynomial in $\F_q[x]$ and assume that $\pi_1^{-1}(x) 
\pmod{x^{p^k}-x}$ is a linearized polynomial over $\F_q[x]$. This is equivalent to saying that 
\[
\pi_1^{-1}(x) = \sum_{i} \alpha_i x^{s_i(p^k-1)+p^{r_i}}.
\]
Therefore,
\[
\pi_1^{-1}\left( \frac{t}{b'} \right) = \sum_{i}  \frac{\alpha_i t^{p^{r_i}}}{ (b')^{s_i(p^k-1)+p^{r_i}}} = \sum_{i} \beta_i t^{p^{r_i}} =:L_1(t),
\]
where 
\[
\beta_i=\frac{\alpha_i}{ (b')^{s_i(p^k-1)+p^{r_i}}} .
\]
Then \eqref{MainEqT1} reduces to
\begin{align*}
    \cW_{f_1}(-a) &= p^{m-k} \sum_{\gamma \in \F_{p^k}} \zeta_p^{B(\gamma)} \sum_{t \in \F_{p^k}^*} \zeta_p^{\Tr_m(a' L_1(t))+\Tr_k (\gamma t)}\\
    &= p^{m-k} \sum_{\gamma \in \F_{p^k}} \zeta_p^{B(\gamma)} \sum_{t \in \F_{p^k}^*} \zeta_p^{\Tr_m( \sum_{i} a'\beta_i t^{p^{r_i}})+\Tr_k (\gamma t)}\\
    &= p^{m-k} \sum_{\gamma \in \F_{p^k}} \zeta_p^{B(\gamma)} \sum_{t \in \F_{p^k}^*} \zeta_p^{\Tr_m( t\sum_{i} (a'\beta_i)^{p^{k-r_i}})+\Tr_k (\gamma t)}\\
    &= p^{m-k} \sum_{\gamma \in \F_{p^k}} \zeta_p^{B(\gamma)} \sum_{t \in \F_{p^k}^*} \zeta_p^{\Tr_k( t \Tr_k^m(\sum_{i} (a'\beta_i)^{p^{k-r_i}}))+\Tr_k (\gamma t)}\\
     &= p^{m-k} \sum_{\gamma \in \F_{p^k}} \zeta_p^{B(\gamma)} \sum_{t \in \F_{p^k}^*} \zeta_p^{\Tr_k ((\delta_1+\gamma) t)},\\
     &= p^{m-k} \sum_{\gamma \in \F_{p^k}} \zeta_p^{B(\gamma)} \sum_{t \in \F_{p^k}} \zeta_p^{\Tr_k ((\delta_1+\gamma) t)}\\
     &= q\zeta_p^{B(-\delta_1)},
\end{align*}
where
\[
    \delta_1 = \Tr_k^m \left(\sum_{i} (a'\beta_i)^{p^{k-r_i}}\right).
\]
This completes the proof.
\end{proof}

\begin{cor}\label{co1}
If permutation $\pi_1^{-1}(x)$ (considered as a polynomial in $\F_q[x]$) reduced modulo $x^{p^k}-x$, is a linearized polynomial over 
$\F_q[x]$ then $\cS_1$ in \eqref{part1} is a normal bent partition of $\F_{q^2}$ of Type I. All bent functions obtained from $\cS_1$ are 
regular. Conversely, if $q>3$ then every $p$-ary bent function that is constant on the elements of $\cS_1$ is such that for every 
$j\in\F_p$, $D_j$ contains exactly $p^{k-1}$ elements of     $\cS_1\setminus\{U\}$. 
\end{cor}

\begin{proof}
Use Definition~\ref{nbp} and take any function $f:\F_{q^2}\mapsto\F_p$ such that $f(U)\equiv c\in \F_p$ and for every $j\in\F_p$, $D_j$ 
contains exactly $p^{k-1}$ elements of $\cS_1 \backslash \{U\}$. Wlog assume that $f(\cA_1(0))=0$. Define function $g:\F_{q^2}\mapsto\F_q$ 
to be constant on the elements of $\cU_m$ by setting $g(U)\equiv 0$ and $g(U^*_j)\equiv j$ for $j \in \F_q$. By the definition of $U_j$, 
this actually means that $g(z)=g_1(z)$. Also, define function $B:\F_{p^k}\mapsto\F_p$ such that $B(\gamma)=f(z)$ taking any 
$z\in\cA_1(\gamma)$, then $B$ is balanced and $B(0)=0$. For $B$ and $g$ defined this way, we have that 
$f(z)=B(\Tr^m_k(g_1(z)))+c(1-(z+z^q)^{q-1})$ that is a function of form \eqref{f1}. Therefore, it suffices to show bentness of $f_1$ with 
any balanced $B$ with $B(0)=0$ to conclude that $\cS_1$ is a bent partition of $\F_{q^2}$. 

Converse statement follows from the Nyberg's lemma (see, e.g., \cite[Lemma~1(i)]{AM22}). 
\end{proof}

\begin{rmk}
Theorem~\ref{th:f1} and Corollary~\ref{co1} also hold for $p=2$ taking 
\[g_1(z)= (z+z^q) (\pi_1(\beta^qz+\beta z^q))^{q-2},\] 
where $\beta\in\F_{q^2}$ is such that $\beta^q=\beta+1$. The sets $U$ and $U_j$, $j \in \F_q$ are defined as follows 
\begin{align*}
    U&= \{ z \in \F_{q^2} \mid \beta^qz+\beta z^q =0\} =\beta \F_q,\\
    U_j &= \left \{z \in \F_{q^2} \mid (z+z^q)=  j \pi_1(\beta^qz+\beta z^q)  \right \}.
\end{align*}
For any $j \in \F_q$, the bijective map from $\F_q$ to $U_j$ is defined as $u \mapsto u+j\beta \pi_1(u)$. The proof exactly follows the one 
of Theorem~\ref{th:f1}. 
\end{rmk}

\section{Normal Bent Partitions of Type I (Construction 2)}
Let $g_2$ be a function from $\F_{q^2}$ to $\F_q$ given by
\[
g_2(z)= \left(\pi_2\left(\frac{z-z^q}{\beta}\right) \right)^{q-2} (z+z^q),
\]
where $\pi_2$ is a permutation over $\F_q$ such that $\pi_2(0)=0$ and $\beta \in \F_{q^2}$ such that $\beta^q=-\beta$. Let $V:=\F_q$, then 
it is easy to observe that $g_2(V)=0$. For any element $j \in \F_q$, define the set 
\[
V_j := \left \{z \in \F_{q^2} \mid z+z^q=  j\pi_2\left(\frac{z-z^q}{\beta}\right)  \right \},
\]
and $V_j^*=  V_j \backslash \{0\}$. It is easy to observe that $V \cap V_j = \{0\}$ for all $j \in \F_q$ and $g_2(V^*_j)=j$. Also, notice 
that for any $j \in \F_q$, the map 
\begin{equation}\label{u2}
 u \mapsto \frac{1}{2}(u \beta +j \pi_2(u)) 
\end{equation}
is a bijection from $\F_q$ to $V_j$. Thus, $\lvert V_j \rvert = q$ for all $j \in \F_q$ and $V_{j_1} \cap V_{j_2}= \{0\}$ for all $j_1, j_2 
\in \F_q$, $j_1 \neq j_2$. Thus, the set 
\[
\cV_m= \{V, V_j \mid j \in \F_q\}
\]
is a normal $m$-partition of $\F_{q^2}$. Now we combine the elements of $\cV_m$ to form a normal partition of $\F_{q^2}$. For each $\gamma 
\in \F_{p^k}$, define the sets 
\[
\cA_2(\gamma):= \bigcup_{\substack{j \in \F_q \\ \Tr_k^m(j)= \gamma}} V_j^*.  
\]
Thus, the set
\begin{equation}\label{part2}
    \cS_2=\{V, \cA_2(\gamma) \mid \gamma \in \F_{p^k}\}
\end{equation}
is a normal partition of $\F_{q^2}$. One may note that $\lvert \cA_2(\gamma) \rvert= p^{m-k}(q-1)$ for all $\gamma \in \F_{p^k}$. 

Define function $f_2$ from $\F_{q^2}$ to $\F_p$ as 
\begin{equation} \label{f2}
\begin{split}
     f_2(z) &= B(\Tr_k^m(g_2(z))) + c(1-(\beta^{-1}(z-z^q))^{q-1}) =
     \begin{cases}
         c~&~\mbox{if}~z\in V\\
         B(\Tr_k^m(g_2(z)))~~&~\mbox{if}~z\in \F_{q^2} \backslash V,
     \end{cases}
\end{split}
\end{equation}
where $B: \F_{p^k} \rightarrow \F_p$ is any balanced function and $B(0)=0$ and $c \in \F_p$. 

\begin{thm}\label{th:f2}
For any function $f_2$ given in~\eqref{f2} and any $a \in \F_{q^2}$ with $a+a^q =0$, the corresponding Walsh transform coefficient 
$\cW_{f_2}(a)=q \zeta_p^{c}$ and, thus, $f^*_2(a)=c$. Also, if $a \in \F^*_q$ then $\cW_{f_2}(a)=q$ and, thus, $f^*_2(a)=0$. Further, for 
any $a \in \F_{q^2} \backslash \f_q$ with $a+a^q\neq 0$, consider two cases. 
\renewcommand{\theenumi}{\roman{enumi}}
\renewcommand{\labelenumi}{(\theenumi)}
\begin{enumerate}
\item\label{it:3} If $B()$ is linear on $\F_{p^k}$ then $f_2$ is bent if and only if $B(\gamma)=\Tr_k(t'\gamma)$ for some 
    $t'\in\F^*_{p^k}$ and $\gamma\in\F_{p^k}$. In this case, $f_2$ is a regular bent function and $f^*_2(-a)=\Tr_m \left(a' 
    \pi_2^{-1}\left(\frac{-t'}{b'}\right) \right)$. Moreover, vectorial function $F_2(z)=\Tr_k^m(g_2(z))+ c(1-(\beta^{-1}(z-z^q))^{q-1})$ 
    is bent and all its components belong to the copmleted original class of Maiorana-McFarland bent functions with associated subspace 
    $V$. 
\item\label{it:4} If $\pi_2^{-1}(x)$, considered as a polynomial in $\F_q[x]$, is such that $\pi_2^{-1}(x) \pmod{x^{p^k}-x}$ is a 
    linearized polynomial over $\F_q[x]$ (say, $\pi_2^{-1}(x) = \sum_{i} \alpha_i x^{s_i(p^k-1)+p^{r_i}}$) then $f_2$ is a regular bent 
    function and $f^*_2(-a)=B(-\delta_2)$, where $\delta_2= \Tr_k^m\left(\sum_{i} \left(\frac{a' \alpha_i}{ (b')^{s_i(p^k-1)+p^{r_i}}} 
    \right)^{p^{k-r_i}}\right)$. 
\end{enumerate}
In both cases, $a' = \frac{\beta(a-a^q)}{2}\in \F_q $ and $b' = \frac{a+a^q}{2} \in \F_q$.
\end{thm}

\begin{proof}
Start with calculating the Walsh transform at $a=0$. 
\begin{equation*}
    \begin{split}
        \cW_{f_2}(0) = \sum_{z \in \F_{q^2}} \zeta_p^{f_2(z)} &=  \sum_{z \in V} \zeta_p^c + \sum_{\gamma \in \F_{p^k}} \sum_{z \in \cA_2(\gamma)} \zeta_p^{B(\Tr_k^m(g_2(z)))}\\
        &= \lvert V \rvert \zeta_p^c + \sum_{\gamma \in \F_{p^k}}  \lvert \cA_2(\gamma) \rvert \zeta_p^{B (\gamma)} \\
        &= \lvert V \rvert \zeta_p^c + p^{m-k}(q-1) \sum_{\gamma \in \F_{p^k}} \zeta_p^{B (\gamma)} \\
        &=q\zeta_p^c.
    \end{split}
\end{equation*} 
Now, for any nonzero $a \in \F_{q^2}$, we have
\begin{align*}
        \cW_{f_2}(-a) &= \sum_{z \in \F_{q^2}} \zeta_p^{B(\Tr_k^m(g_2(z)))+ c(1-(\beta(z-z^q))^{q-1}) +\Tr_{2m}(az)}\\
        &=  \sum_{z \in V} \zeta_p^{c+\Tr_{2m}(az)} + \sum_{\gamma \in \F_{p^k}} \sum_{z \in \cA_2(\gamma)} \zeta_p^{B(\Tr_k^m  (g_2(z)))+\Tr_{2m}(az)}\\
        &= \zeta_p^c \sum_{z \in \F_q} \zeta_p^{\Tr_{2m}(az)} + \sum_{\gamma \in \F_{p^k}} \zeta_p^{B(\gamma)} \sum_{z \in \cA_2(\gamma )} \zeta_p^{\Tr_{2m}(az)} \\
      &= \zeta_p^c \sum_{z \in \F_q} \zeta_p^{\Tr_{2m}(az)} + \sum_{\gamma \in \F_{p^k}} \zeta_p^{B(\gamma)} \sum_{\substack{j \in \F_q \\ \Tr_k^m(j)=\gamma }}  \sum_{z \in V_j^*} \zeta_p^{\Tr_{2m}(az)} \\
      &\stackrel{\eqref{u2}}{=} \zeta_p^c \sum_{z \in \F_q} \zeta_p^{\Tr_{2m}(az)} + \sum_{\gamma \in \F_{p^k}} \zeta_p^{B(\gamma)}  \sum_{\substack{j \in \F_q \\ \Tr_k^m(j)=\gamma }}  \sum_{u \in \F_q^*} \zeta_p^{\Tr_{2m}\left(\frac{a}{2} (u \beta +j \pi_2(u))\right)}\\
       &= \zeta_p^c \sum_{z \in \F_q} \zeta_p^{\Tr_{2m}(az)} + \sum_{\gamma \in \F_{p^k}} \zeta_p^{B(\gamma)}  \sum_{\substack{j \in \F_q \\ \Tr_k^m(j)=\gamma }}  \sum_{u \in \F_q^*} \zeta_p^{\Tr_m \left(u \beta \left(\frac{a-a^q}{2} \right)+j \pi_2(u) \left(\frac{a+a^q}{2} \right) \right)}\\
       &= \zeta_p^c \sum_{z \in \F_q} \zeta_p^{\Tr_{2m}(az)} + \sum_{\gamma \in \F_{p^k}} \zeta_p^{B(\gamma)} \sum_{\substack{j \in \F_q \\ \Tr_k^m(j)=\gamma }}  \sum_{u \in \F_q^*} \zeta_p^{\Tr_m \left(a' u+jb'\pi_2(u) \right)},
    \end{align*}
where $\frac{\beta(a-a^q)}{2}=a' \in \F_q $ and $\frac{a+a^q}{2}  =b' \in \F_q$. Now, we consider two cases, namely, $a \in \F_q^*$ and $a \in \F_{q^2} \backslash \f_q$.\\
\textbf{Case 1.} Let $a \in \F_q^*$. In this case, $a'=0$ and $b'=a$. For each $\gamma \in \F_{p^k}$, fix $j_\gamma \in \F_q$ such that 
$\Tr_k^m(j_\gamma)= \gamma $. Let $j= j_\gamma + v$ for some $v \in \ker(\Tr_k^m)$. Then, we have 
\begin{align*}
        \cW_{f_2}(-a) &= \zeta_p^c \sum_{z \in \F_q} \zeta_p^{\Tr_m((a+a^q)z)} + \sum_{\gamma \in \F_{p^k}} \zeta_p^{B(\gamma)} \sum_{\substack{j \in \F_q \\ \Tr_k^m(j)=\gamma }}  \sum_{u \in \F_q^*} \zeta_p^{\Tr_m (ja\pi_2(u))}\\
        &=\zeta_p^c \sum_{z \in \F_q} \zeta_p^{\Tr_m(2az)} + \sum_{\gamma \in \F_{p^k}} \zeta_p^{B(\gamma)} \sum_{v \in \ker(\Tr_k^m)}  \sum_{u \in \F_q^*} \zeta_p^{\Tr_m ((j_\gamma +v)a\pi_2(u))}\\
        &= \sum_{\gamma \in \F_{p^k}} \zeta_p^{B(\gamma)}  \sum_{u \in \F_q^*} \zeta_p^{\Tr_m (j_\gamma a\pi_2(u))}  \sum_{v \in \ker(\Tr_k^m)} \zeta_p^{\Tr_k (\Tr_k^{m}(av\pi_2(u)))}\\
        &= p^{m-k}\sum_{\gamma \in \F_{p^k}} \zeta_p^{B(\gamma)}  \sum_{\substack{u \in \F_q^* \\ a\pi_2(u) \in \F_{p^k}^*}} \zeta_p^{\Tr_k (\gamma a\pi_2(u))}  \\
        &= p^{m-k}(p^k-1) +p^{m-k}\sum_{\gamma \in \F_{p^k}^*} \zeta_p^{B(\gamma)}  \sum_{\substack{u \in \F_q^* \\ a\pi_2(u) \in \F_{p^k}^*}} \zeta_p^{\Tr_k (\gamma a\pi_2(u))} \\
        &= p^{m-k}(p^k-1)-p^{m-k}\sum_{\gamma \in \F_{p^k}^*} \zeta_p^{B(\gamma)}\\
        &= p^{m}-p^{m-k}\sum_{\gamma \in \F_{p^k}} \zeta_p^{B(\gamma)}\\
        &= q.
\end{align*}

\textbf{Case 2.} Let $a \in \F_{q^2} \backslash \f_q$. In this case, we have $a' \neq 0$. Now we consider two subcases, namely, $b'=0$ and $b' \neq 0$.\\

\textbf{Subase 2.1.} Let $b'= 0$. In this case, we have 
\begin{align*}
        \cW_{f_2}(-a) &= \zeta_p^c \sum_{z \in \F_q} \zeta_p^{\Tr_{2m}(az)} + \sum_{\gamma \in \F_{p^k}} \zeta_p^{B(\gamma)} \sum_{\substack{j \in \F_q \\ \Tr_k^m(j)=\gamma }}  \sum_{u \in \F_q^*} \zeta_p^{\Tr_m \left(a' u \right)}\\
        &= \zeta_p^c \sum_{z \in \F_q} \zeta_p^{\Tr_m((a+a^q)z)} -p^{m-k} \sum_{\gamma \in \F_{p^k}} \zeta_p^{B(\gamma)} \\
        &=q \zeta_p^c.
    \end{align*}

\textbf{Subase 2.2.} Let $b' \neq 0$. In this case, we have 
\begin{align*}
        \cW_{f_2}(-a) &= \zeta_p^c \sum_{z \in \F_q} \zeta_p^{\Tr_{2m}(az)} + \sum_{\gamma \in \F_{p^k}} \zeta_p^{B(\gamma)} \sum_{\substack{j \in \F_q \\ \Tr_k^m(j)=\gamma }}  \sum_{u \in \F_q^*} \zeta_p^{\Tr_m \left(a' u+jb'\pi_2(u) \right)}\\
        &= \zeta_p^c \sum_{z \in \F_q} \zeta_p^{\Tr_m((a+a^q)z)} + \sum_{\gamma \in \F_{p^k}} \zeta_p^{B(\gamma)} \sum_{v \in \ker(\Tr_k^m)}  \sum_{u \in \F_q^*} \zeta_p^{\Tr_m (a' u+(j_\gamma +v)b'\pi_2(u))}\\
        &= \sum_{\gamma \in \F_{p^k}} \zeta_p^{B(\gamma)}  \sum_{u \in \F_q^*} \zeta_p^{\Tr_m (a' u+j_\gamma b'\pi_2(u))}  \sum_{v \in \ker(\Tr_k^m)} \zeta_p^{\Tr_k (\Tr_k^{m}(v b' \pi_2(u)))}\\
        &= p^{m-k}\sum_{\gamma \in \F_{p^k}} \zeta_p^{B(\gamma)}  \sum_{\substack{u \in \F_q^* \\ b'\pi_2(u) \in \F_{p^k}^*}} \zeta_p^{\Tr_m (a' u)+\Tr_k (\gamma b'\pi_2(u))}.
    \end{align*}
Since $\pi_2$ is bijective on $\F_q$, the map $u \mapsto b'\pi_2(u)$ is bijective on $\F_q$. Thus, there exist precisely $p^k-1$ element $u 
\in \F_q^*$ such that $b'\pi_2(u) \in \F_{p^k}^*$. Let $b' \pi_2(u)= t$ then $u=\pi_2^{-1}\left(\frac{t}{b'}\right)$, where $\pi_2^{-1}$ is 
the compositional inverse of $\pi_2$. Thus, we have 
\begin{equation}\label{MainEqT2}
        \cW_{f_2}(-a) = p^{m-k}\sum_{\gamma \in \F_{p^k}} \zeta_p^{B(\gamma)}  \sum_{t \in \F_{p^k}^*} \zeta_p^{\Tr_m \left(a' \pi_2^{-1}\left(\frac{t}{b'}\right)\right)+\Tr_k (\gamma t)}.
    \end{equation}

In particular, if $B$ is linear on $\F_{p^k}$ then
\begin{align*}
\cW_{f_2}(-a)&=p^{m-k} \sum_{t \in \F_{p^k}^*} \zeta_p^{\Tr_m \left(a' \pi_2^{-1}\left(\frac{t}{b'}\right) \right)}\sum_{\gamma \in \F_{p^k}} \zeta_p^{B(\gamma)+\Tr_k (\gamma t)}\\
&=q \sum_{\substack{t \in \F_{p^k}^*\\B(\gamma)\equiv -\Tr_k (\gamma t)}} \zeta_p^{\Tr_m \left(a' \pi_2^{-1}\left(\frac{t}{b'}\right) \right)}\\
&=q \zeta_p^{\Tr_m \left(a' \pi_2^{-1}\left(\frac{-t'}{b'}\right) \right)},
\end{align*}
if $B(\gamma)=\Tr_k(t'\gamma)$ for some $t'\in\F^*_{p^k}$ and $\gamma\in\F_{p^k}$. Obviously, $|\cW_{f_2}(-a)|\neq q$ if $B$ is not of the 
trace form. From \eqref{MainEqT2} it also follows that taking $B(\gamma)=\Tr_k(t'\gamma)$ for any $t'\in\F^*_{p^k}$ we have 
$\cW_{f_2}(-a)=q \zeta_p^{\Tr_m \left(a' \pi_2^{-1}\left(\frac{-t'}{b'}\right) \right)}$. Therefore, vectorial function $F_2(z)= 
\Tr_k^m(g_2(z))+ c(1-(\beta^{-1}(z-z^q))^{q-1})$ is bent and all its components are regular (each component function $\Tr_k(t'F_2(z))$ has 
the form of \eqref{f1} with $c\Tr_k(t')$ in place of $c$). 

In the general case with balanced $B$, consider $\pi_2^{-1}(x)$ as a polynomial in $\F_q[x]$ and assume that $\pi_2^{-1}(x) 
\pmod{x^{p^k}-x}$ is a linearized polynomial over $\F_q[x]$. This is equivalent to saying that 
\[
\pi_2^{-1}(x) = \sum_{i} \alpha_i x^{s_i(p^k-1)+p^{r_i}}.
\]
Therefore,
\[
\pi_2^{-1}\left( \frac{t}{b'} \right) = \sum_{i} \frac{\alpha_i t^{p^{r_i}}}{ (b')^{s_i(p^k-1)+p^{r_i}}} = \sum_{i} \beta_i t^{p^{r_i}} =:L_2(t),
\]
where 
\[
\beta_i=\frac{\alpha_i}{ (b')^{s_i(p^k-1)+p^{r_i}}} .
\]
Then \eqref{MainEqT2} reduces to
\begin{align*}
    \cW_{f_1}(-a) &= p^{m-k} \sum_{\gamma \in \F_{p^k}} \zeta_p^{B(\gamma)} \sum_{t \in \F_{p^k}^*} \zeta_p^{\Tr_m(a' L_2(t))+\Tr_k (\gamma t)}\\
    &= p^{m-k} \sum_{\gamma \in \F_{p^k}} \zeta_p^{B(\gamma)} \sum_{t \in \F_{p^k}^*} \zeta_p^{\Tr_m( \sum_{i} a'\beta_i t^{p^{r_i}})+\Tr_k (\gamma t)}\\
    &= p^{m-k} \sum_{\gamma \in \F_{p^k}} \zeta_p^{B(\gamma)} \sum_{t \in \F_{p^k}^*} \zeta_p^{\Tr_m( t\sum_{i} (a'\beta_i)^{p^{k-r_i}})+\Tr_k (\gamma t)}\\
    &= p^{m-k} \sum_{\gamma \in \F_{p^k}} \zeta_p^{B(\gamma)} \sum_{t \in \F_{p^k}^*} \zeta_p^{\Tr_k( t \Tr_k^m(\sum_{i} (a'\beta_i)^{p^{k-r_i}}))+\Tr_k (\gamma t)}\\
     &= p^{m-k} \sum_{\gamma \in \F_{p^k}} \zeta_p^{B(\gamma)} \sum_{t \in \F_{p^k}^*} \zeta_p^{\Tr_k ((\delta_2+\gamma) t)},\\
     &= p^{m-k} \sum_{\gamma \in \F_{p^k}} \zeta_p^{B(\gamma)} \sum_{t \in \F_{p^k}} \zeta_p^{\Tr_k ((\delta_2+\gamma) t)}\\
     &= q\zeta_p^{B(-\delta_2)},
\end{align*}
where 
\[
\delta_2 = \Tr_k^m \left(\sum_{i} (a'\beta_i)^{p^{k-r_i}}\right).
\]
This completes the proof.
\end{proof}

\begin{cor}\label{co2}
If permutation $\pi_2^{-1}(x)$ (considered as a polynomial in $\F_q[x]$) reduced modulo $x^{p^k}-x$, is a linearized polynomial over 
$\F_q[x]$ then $\cS_2$ in \eqref{part2} is a normal bent partition of $\F_{q^2}$ of Type I. All bent functions obtained from $\cS_2$ are 
regular. Conversely, if $q>3$ then every $p$-ary bent function that is constant on the elements of $\cS_2$ is such that for every 
$j\in\F_p$, $D_j$ contains exactly $p^{k-1}$ elements of     $\cS_2\setminus\{U\}$. 
\end{cor}

\begin{rmk}
Theorem~\ref{th:f2} and Corollary~\ref{co2} also hold for $p=2$ taking 
\[g_2(z)=  (\beta^qz+\beta z^q)(\pi_1(z+z^q))^{q-2},\]
The sets $V$ and $V_j$, $j \in \F_q$ are defined as follows 
\begin{align*}
    V&= \{ z \in \F_{q^2} \mid z+z^q=0\} = \F_q,\\
    V_j &= \left \{z \in \F_{q^2} \mid (\beta^qz+\beta z^q)=  j \pi_1(z+z^q)  \right \}.
\end{align*}
For any $j \in \F_q$, the bijective map from $\F_q$ to $V_j$ is defined as $u \mapsto u\beta+j \pi_1(u)$. The proof exactly follows the one 
of Theorem~\ref{th:f2}.
\end{rmk}

\section{Examples of Bent Partitions of Type I}
Take $p$ odd and let $\beta \in \F_{q^2}$ such that $\beta^q=-\beta$. Define a map $\psi$ from $\F_q \times \F_q \rightarrow \F_{q^2}$ as
\[
\psi(x,y) = \frac{x+\beta y}{2} =:z.
\]
It is straightforward to see that $\psi$ is a bijective map and its inverse is given by
\[
\psi^{-1}(z) = \left(z+z^q, \frac{z-z^q}{\beta} \right).
\]

Using the above defined bijective map $\psi$, the bivariate representation of our Construction~1 is given by
\begin{align*}
    g_1(x,y) &= ((\pi_1(x))^{q-2}y\\
    U &=(0,y)~\mbox{and}~U_j=(x,j \pi_1(x))~\mbox{for all}~j\in \F_q\\
    \cU_m &= \{U, U_j \mid j \in \F_q\}\\
    f_1(x,y) &= B(\Tr_k^m((\pi_1(x))^{q-2}y))+ c(1-x^{q-1}).
\end{align*}

For any $\gamma \in \F_q$, 
\[
 \cA_1(\gamma)= \bigcup_{\substack{j \in \F_q \\ \Tr_k^m(j)= \gamma}} U_j^*,
\]
and 
\begin{equation*}
    \cS_1=\{U, \cA_1(\gamma) \mid \gamma \in \F_{p^k}\}
\end{equation*}

Consider two special cases:
\begin{enumerate}
    \item Assume $k=m$ and $\pi_1$ is the identity map, i.e., $\pi_1(x)=x$ for all $x \in \F_q$. Then
    \begin{enumerate}
        \item bent partition $\cS_1$ reduces to $\cU_m$ and is called the {\bf Desarguesian spread}.
     \item function $f_2$ reduces to
    \begin{equation*}
         f_1(x,y) = B(x^{q-2}y)+ c(1-x^{q-1}).
    \end{equation*}
    Dillon~\cite{D74} showed that Boolean (also generalized for thee $p$-ary case in \cite{LiLu14}) bent functions obtained from the 
    Desarguesian spread have the above form and called them {\bf $\cPS_{ap}$ bent functions}. 
    \end{enumerate}
    \item Assume $\pi_1(x)=x^e$ such that $\gcd(e, q-1)=1$ and $e \equiv p^{\ell} \pmod{p^k-1}$. Then
    \begin{enumerate}
        \item according to \cite{AKM26}, partition $\cS_1$ is called a {\bf generalized Desarguesian spread}. 
        \item function $f_1$ reduces to 
        \[
         f_1(x,y) = B(\Tr_k^m((x^e)^{q-2}y))+ c(1-x^{q-1})
        \]
        The authors in \cite{AKM23} showed that bent functions obtained from a generalized Desarguesian spread have the above form and 
        called them {\bf generalized $\cPS_{ap}$ bent functions}. 
    \end{enumerate}
\end{enumerate}

\begin{rmk}
If $k=m$ and $\pi_1$ is a linear permutation of $\F_q$, then the bijection from $\F_q$ to $U_j$ given in \eqref{u1}, is a linear map and 
hence, $\cU_m$ is a spread of $\F_{q^2}$. Consequently, all bent functions obtained from the normal bent partition $\cS_1=\cU_m$ belong to 
the partial spread class. 
\end{rmk} 

\begin{rmk}
We use the term (partial) spread if elements of a partition have a subspace structure in $\F_{q^2}$. It is easy to observe that if we 
choose $\pi_1(x)=x^e$ such that $\gcd(e, q-1)=1$ and $e \equiv p^{\ell} \pmod{p^k-1}$ then subsets $U_j$ (and also elements of $\cS_1$) are 
not necessarily subspaces in $\F_{q^2}$. Therefore, using the term  generalized Desarguesian {\bf spread} is not suitable here. Instead, we 
suggest using the term {\bf generalized Desarguesian spread} for the case when $\pi_1$ is a linear permutation. In the particular case, 
when $k=m$ and $\pi_1$ is the identity map, we obtain the Desarguesian spread. 
\end{rmk}

Using the above defined bijective map $\psi$, the bivariate representation of our Construction~2 is given by
\begin{align*}
    g_2(x,y) &= x(\pi_2(y))^{q-2}\\
    V &=(x,0)~\mbox{and}~ V_j=(j \pi_2(y),y)~\mbox{for all}~j\in \F_q\\
    \cV_m &= \{V, V_j \mid j \in \F_q\}\\
    f_2(x,y) &= B(\Tr_k^m(x(\pi_2(y))^{q-2}))+ c(1-y^{q-1}).
\end{align*}
For any $\gamma \in \F_q$
\[
 \cA_2(\gamma)= \bigcup_{\substack{j \in \F_q \\ \Tr_k^m(j)= \gamma}} V_j^*,
\]
and 
\begin{equation*}
    \cS_2=\{V, \cA_2(\gamma) \mid \gamma \in \F_{p^k}\}
\end{equation*}

Consider two special cases:
\begin{enumerate}
    \item Assume $k=m$ and $\pi_2$ is the identity map, i.e., $\pi_2(y)=y$ for all $y \in \F_q$. Then
    \begin{enumerate}
        \item bent partition $\cS_2$ reduces to $\cV_m$, and is the Desarguesian spread.
     \item function $f_2$ reduces to
    \begin{equation*}
         f_2(x,y) = B(xy^{q-2})+ c(1-y^{q-1}).
    \end{equation*}
    Dillon~\cite{D74} showed that Boolean (also true for the $p$-ary, see \cite{LiLu14}) bent functions obtained from the Desarguesian 
    spread have the above form and called them $\cPS_{ap}$ bent functions. 
    \end{enumerate}
    \item Assume $\pi_2(x)=x^d$ such that $\gcd(d, q-1)=1$ and $d \equiv p^{\ell} \pmod{p^k-1}$. Then
    \begin{enumerate}
        \item according to \cite{AKM26}, partition $\cS_2$ is called a generalized Desarguesian spread. 
        \item function $f_2$ reduces to 
        \[
         f_2(x,y) = B(\Tr_k^m(x(y^d)^{q-2}))+ c(1-y^{q-1})
        \]
        The authors in \cite{AKM23} showed that bent functions obtained from a generalized Desarguesian spread have the above form and 
        called them generalized $\cPS_{ap}$ bent functions. 
    \end{enumerate}
\end{enumerate}

\section{Vectorial Bent Functions and Duality}

In~\cite{CMP18}, the concept of duality for vectorial bent functions was introduced.

\begin{defn}
Vectorial bent function $F: \F_{p^n} \mapsto \F_{p^m}$ is called vectorial dual-bent if the set consisting of the duals of the component functions of $F$, forms the set of component functions of another vectorial bent function from $\F_{p^n}$ to $\F_{p^m}$, denoted by $F^*$, which is called a vectorial dual of $F$. 
\end{defn}

\begin{defn}
Let $n=2m$ be a positive integer and let $0<k\leq m$. A vectorial dual-bent function $F: \F_{p^n} \mapsto \F_{p^k}$ is said to satisfy {\bf 
Condition A}. If 
\begin{enumerate}
    \item all component functions $F_t= \Tr_k(tF)$ with $t \in \F_{p^k}^*$, are regular, or all component functions are weakly regular 
        but not regular 
    \item for any $t_1, t_2 \in \F_{p^k}^*$ with $t_1+t_2 \neq 0$, the equality
    \[
    (F_{t_1})^*+ (F_{t_2})^*=(F_{t_1+t_2})^*
    \]
    holds (that is, the sum of the duals equals the dual of the sum).
\end{enumerate}
\end{defn}

Let $F: \F_{p^{2m}} \mapsto \F_{p^k}$ be a vectorial function given by
\begin{equation}
    \label{vectF}
F(z)=\Tr_k^m(g_1(z))+ c(1-(z+z^q)^{q-1}).
\end{equation}
Note that partition \eqref{part1} with $U$ joined with $\cA_1(c)$, is the preimage set partition of $F$.

Taking any $t\in\F^*_{p^k}$, the corresponding component function of $F$ is
\begin{equation}
    \label{compF}
    \begin{split}
        F_t(z) &= \Tr_k \left( t \Tr_k^m \left( \frac{z-z^q}{\beta} (\pi_1(z+z^q))^{q-2} \right) + ct(1-(z+z^q)^{q-1}) \right)\\
    &= \Tr_m \left( \frac{t(z-z^q)}{\beta} (\pi_1(z+z^q))^{q-2} \right) + c\Tr_k \left(t(1-(z+z^q)^{q-1}) \right)\\
    &=
    \begin{cases}
         c\Tr_k(t) ~\quad~&\mbox{if}~z\in U,\\
         \Tr_m \left( \frac{t(z-z^q)}{\beta} (\pi_1(z+z^q))^{q-2} \right)~\quad~&\mbox{if}~z\in \F_{q^2} \backslash U
    \end{cases}
    \end{split}
\end{equation}
and has the form of \eqref{f1} with $B(\gamma)=\Tr_k(t\gamma)$ and $c\Tr_k(t)$ taken for $c$. By Theorem~\ref{th:f1}, Case~(\ref{it:1}), 
$F_t$ is bent and its dual is given by 
\begin{equation}
    \label{compFstar}
    \begin{split}
         (F_t)^*(z) &=\Tr_m \left(-\frac{z+z^q}{2} \pi_1^{-1} \left( \left(\frac{\beta(z-z^q)}{2t}\right)^{q-2} \right) \right) + c \Tr_k(t(1-(z-z^q)^{q-1})) \\
 &=
    \begin{cases}
         c\Tr_k(t) ~\quad~&\mbox{if}~z\in \F_q,\\
        \Tr_m \left(-\frac{z+z^q}{2} \pi_1^{-1}\left(\frac{2t}{\beta(z-z^q)}\right) \right)~\quad~&\mbox{if}~z\in \F_{q^2} \backslash \F_q,
    \end{cases}
    \end{split}
\end{equation}
It is easy to observe that
\begin{align*}
    (F_t)^*\left( \frac{2tz}{\beta} \right) &=\Tr_m \left(-\frac{t(z-z^q)}{\beta} \pi_1^{-1}\left((z+z^q)^{q-2}\right) \right) + c \Tr_k(t(1+(z+z^q)^{q-1}))\\
    &=\Tr_k \left( -t\Tr_k^m \left(\frac{z-z^q}{\beta} \pi_1^{-1}\left((z+z^q)^{q-2}\right) \right) +  ct(1+(z+z^q)^{q-1}) \right).
\end{align*}
Thus, for any fixed $t \in \F_{p^k}^*$, the dual of the component function $F_t(z)$ is the component function of the vectorial function 
$G\circ L_t$, where 
\[
G(z)= \Tr_k^m \left(-\frac{z-z^q}{\beta} \pi_1^{-1}\left((z+z^q)^{q-2}\right) \right) +  c(1+(z+z^q)^{q-1})\quad\mbox{and}\quad L_t(z)=\frac{\beta}{2t}z.
\]

\begin{rmk}
    Since the vectorial function $G\circ L_t$ varies with $t$, we can not say this to be a vectorial dual of $F$.
\end{rmk}

In the particular case, when $\pi_1^{-1}(x)= \sum_i \alpha_i x^{s_i(p^k-1)+p^{r_i}}$, then from Equation~\eqref{compFstar}, we have
\begin{align*}
         (F_t)^*(z)  &=\Tr_m \left(-\frac{z+z^q}{2} \pi_1^{-1} \left( \left(\frac{\beta(z-z^q)}{2t}\right)^{q-2} \right) \right) + \Tr_k(ct(1-(z-z^q)^{q-1})) \\
         &=\Tr_m \left( -\sum_{i} \alpha_i \frac{z+z^q}{2} \left(\frac{\beta(z-z^q)}{2t}\right)^{(q-2)(s_i(p^k-1)+p^{r_i})} \right) + \Tr_k(ct(1-(z-z^q)^{q-1})) \\
         &= -\sum_{i} \Tr_m \left(\alpha_i (z+z^q) \left(\frac{\beta(z-z^q)}{t}\right)^{(q-2)(s_i(p^k-1)+p^{r_i})} \right) + \Tr_k(ct(1-(z-z^q)^{q-1})) \\
        &= -\sum_{i} \Tr_m \left(t (\alpha_i (z+z^q))^{p^{k-r_i}} (\beta(z-z^q))^{p^{k-r_i}(q-2)(s_i(p^k-1)+p^{r_i})} \right) + \Tr_k(ct(1-(z-z^q)^{q-1})) \\
        &= -\sum_{i} \Tr_k \left(t \Tr_k^m \left( (\alpha_i (z+z^q))^{p^{k-r_i}} (\beta(z-z^q))^{p^{k-r_i}(q-2)(s_i(p^k-1)+p^{r_i})} \right) \right)+ \Tr_k(ct(1-(z-z^q)^{q-1})) \\
        &=  -\Tr_k \left(t  \Tr_k^m \left(\sum_{i} (\alpha_i (z+z^q))^{p^{k-r_i}} (\beta(z-z^q))^{p^{k-r_i}(q-2)(s_i(p^k-1)+p^{r_i})} \right) +ct(1-(z-z^q)^{q-1}) \right).
    \end{align*}
We conclude that the vectorial function
\begin{equation}
    \label{dualvectF}
    \begin{split}
        F^*(z) &=  \Tr_k^m \left( \sum_{i} (\alpha_i (z+z^q))^{p^{k-r_i}} (\beta(z-z^q))^{p^{k-r_i}(q-2)(s_i(p^k-1)+p^{r_i})} \right) +c(1-(z-z^q)^{q-1})
    \end{split}
\end{equation}
is such that each of its component function $F_t^*$ is the dual of the component function $F_{-t}$ of $F$.

\begin{prop}
    Let $F$ be a vectorial function given in \eqref{vectF}. If $\pi_1^{-1} (x) = \sum_i\alpha_i x^{s_i(p^k-1)+p^{r_i}}$ then $F$ is vectorial dual-bent function and it satisfies Condition A.
\end{prop}
\begin{proof}
    The vectorial dual-bent property is proven. For any $t_1, t_2 \in \F_{p^k}^*$, with $t_1+t_2 \neq 0$, the equality $(F_{t_1})^*+ (F_{t_2})^*=(F_{t_1+t_2})^*$
    follows directly from the linearity of the trace function.
\end{proof}

\begin{defn}
    Let $\cS = \{U, \cA(\gamma) \mid \gamma \in \F_{p^k}\}$ be a normal bent partition of $\F_{p^{2m}}$ of depth $p^k$. Then $\cS$ is said to satisfy {\bf Condition C}. If
    \begin{enumerate}
        \item $\F_p^* \cA(\gamma) = \cA(\gamma)$ for all $\gamma \in \F_{p^k}$.
        \item all bent functions obtained from $\cS$ are regular, or all bent functions obtained from $\cS$ are weakly regular but not regular.
    \end{enumerate}
    Joining $U$ with some $\cA(\gamma)$, we obtain a bent partition of $\F_{p^{2m}}$ satisfying Condition C.
\end{defn}

By \cite[Theorem~6]{AKM26}, since $F$ (with appropriate $\pi_1$) is a vectorial dual-bent function that satisfies Condition A, then 
preimage set partition defined by $F$, is a bent partition satisfying Condition C. In the following proposition, we verify this directly. 

\begin{prop}
    Take appropriate permutations $\pi_1$ (as described in Corollary~\ref{co1}), then normal bent partition $\cS_1$ defined
    in~\eqref{part1}, satisfies Condition C. 
\end{prop}

\begin{proof}
Part (2) of Condition C is satisfied because of Corollary~\ref{co1}. Further, since $U= \beta \F_q$, then for any $c \in \F_p^*$, $cU=U$. 
Now we show that for any $j \in \F_{p^k}$, $cU_j=U_j$. It is easy to verify that the function 
    \[
    u \mapsto \frac{1}{2}(\pi^{-1}(u)+j \beta u)
    \]
    is a bijection from $\F_q$ to $U_j$. For any  $c \in \F_p^*$, we have
    \[
    cU_j= \left\{ \frac{1}{2}(c\pi^{-1}(u)+j \beta cu) \mid u \in \F_q \right\} = \left\{ \frac{1}{2}(\pi^{-1}(u)+j \beta u) \mid cu \in \F_q \right\} =U_j,
    \]
    where the second equality follows from the fact that
    \[
    \pi_1^{-1}(cu) = \sum_{i} \alpha_i (cu)^{s_i(p^k-1)+p^{r_i}} = c \sum_{i} \alpha_i u^{s_i(p^k-1)+p^{r_i}}=c\pi_1^{-1}(u) .
    \]
This completes the proof.
\end{proof}

\begin{rmk}
In the similar way, it can be proven that function $\Tr_k^m(g_2(z))+ c(1-(\beta^{-1}(z-z^q))^{q-1})$ (with appropriate permutation $\pi_2$, 
as described in Corollary~\ref{co2}) is vectorial dual-bent function and it satisfies Condition A. Therefore, corresponding normal bent 
partition $\cS_2$ defined in~\eqref{part2}, satisfies Condition C. We do not reproduce the proof here since this is an easy exercise. Also, 
the same holds for $p=2$ and slightly modified partitions defined in Remarks~2.3~and~3.3.
\end{rmk}


\begin{thebibliography}{99}
\bibitem{AAKM24} S. Alkan, N. Anbar,  T. Kalayci, W. Meidl, {\it Bent partitions and LP-packings}, IEEE Trans. Inf. Theory 70 (2024) 5365--5375. 

\bibitem{AM22} N. Anbar, W. Meidl, {\it Bent partitions}, Des. Codes Cryptogr. 90 (2022) 1081--1101. 

\bibitem{AKM26}  N. Anbar,  T. Kalayci, W. Meidl, {\it A survey on bent partitions}, Cryptogr. Commun. 2026. 

\bibitem{AKM23} N. Anbar,  T. Kalayci, W. Meidl, {\it On generalized spread bent partitions}, Cryptogr. Commun. 15 (2023) 1217--1234. 

\bibitem{CMP18} A. Çeşmelioğlu, W. Meidl, A. Pott, {\it Vectorial bent functions and their duals,} Linear Algebra Appl. 548 (2018) 305--320.

\bibitem{D74} J. F. Dillon, {\it Elementary Hadamard difference sets}, Ph.D. dissertation, University of Maryland (1974). 

\bibitem{KSW85} P. V. Kumar, R. A. Scholtz, L. R. Welch,  {\it Generalized bent functions and their properties}, J. Combin. Theory Ser. A  40 (1985) 90--107.

\bibitem{R76} O. S. Rothaus, {\it On ``bent" functions,} J. Combinat. Theory A 20(3) (1976) 300-305.

\bibitem{WWF26} J. Wang, Y. Wei, F. W. Fu, {\it Further results on bent partitions}, IEEE Trans. Inf. Theory. 72 (2026) 3447--3459. 
\end{thebibliography}
\end{document}